\documentclass{jgcc}
\pdfoutput=1
\usepackage{lastpage}
\jgccdoi{12}{1}{4}{6541}  
\jgccheading{}{\pageref{LastPage}}{}{}%
{Sep.~15,~2019}{Jun.~23,~2020}{}


\usepackage{amsmath,amsthm,amsfonts,amssymb,soul}
\usepackage{hyperref}

\newcommand{\gp}[1]{\langle #1 \rangle}
\newcommand{\gpr}[2]{{\left\langle #1 \mid #2 \right\rangle}}

\def\ovx{{\overline{x}}}

\DeclareMathOperator{\Mat}{{Mat}}

\def\P{{\mathbf{P}}}
\def\NP{{\mathbf{NP}}}
\def\SSP{{\mathbf{SSP}}}
\def\ZOE{{\mathbf{ZOE}}}

\begin{document}

\title{On the subset sum problem in branch groups}

\author{Andrey Nikolaev}
\address{Stevens Institute of Technology, Hoboken, NJ, 07030, USA}
\email{anikolae@stevens.edu}

\author{Alexander Ushakov}
\address{Stevens Institute of Technology, Hoboken, NJ, 07030, USA}
\email{aushakov@stevens.edu}

\begin{abstract}
We consider a group-theoretic analogue of the classical subset sum problem. In this brief note, we show that the subset sum problem is $\NP$-complete in the first Grigorchuk group. More generally, we show $\NP$-hardness of that problem in weakly regular branch groups, which implies $\NP$-completeness if the group is, in addition, contracting.	
\end{abstract}

\maketitle

\section{Introduction}\label{sec:intro}

The study of discrete optimization problems in groups was initiated in~\cite{Miasnikov-Nikolaev-Ushakov:2014a}, 
where the authors introduced group-theoretic generalizations of the classical knapsack problem 
and its variations, e.g., the subset sum problem and bounded submonoid membership problem. 
In the subsequent papers~\cite{Miasnikov-Nikolaev-Ushakov:2014b} and~\cite{Myasnikov-Nikolaev-Ushakov:2016}, the authors studied generalizations of the Post corresponce problem and classical lattice problems in groups.
The investigation of knapsack-type problems in groups continued in 
papers~\cite{Frenkel-Nikolaev-Ushakov:2014,Konig-Lohrey-Zetzsche:2016,Lohrey-Zetzsche:2016,Ganardi-etal:2018,Misch-Tr:2017, Misch-Tr:2018}.
The computational properties of these problems, 
aside from being interesting in their own right, 
were shown to be closely related to a wide range of well-known geometric and algorithmic properties of groups. 
For instance, the complexity of knapsack-type problems 
in certain groups depends on geometric features 
of a group such as growth, subgroup distortion, and negative curvature. 
The Post correspondence problem in $G$ is closely related to the twisted conjugacy problem in $G$, the equalizer problem in $G$, and a strong version of the word problem. Furthermore, lattice problems are related to the classical subgroup membership problem
and finite state automata. We refer the reader to the aforementioned papers for details.

In this paper, we prove $\NP$-hardness of the subset sum problem in any finitely generated weakly regular branch group. For groups with polynomial time word problem, e.g., the first Grigorchuk group, this implies $\NP$-completeness.

\subsection{Subset sum problem}

Let $G$ be a group generated by a finite set 
$X=\{x_1,\ldots,x_n\}\subseteq G$. Elements in $G$ can be expressed
as products of the generators in $X$ and their inverses.
Hence, we can state the following combinatorial problem.

\medskip

\noindent{\bf The subset sum  problem $\SSP(G,X)$\index{$\SSP(G,X)$}:}
Given words $g_1,\ldots,g_k,g$ over the alphabet ${X\cup X^{-1}}$, decide if
  \begin{equation} \label{eq:SSP-def}
  g = g_1^{\varepsilon_1} \ldots g_k^{\varepsilon_k}
  \end{equation}
in the group $G$ for some $\varepsilon_1,\ldots,\varepsilon_k \in \{0,1\}$.

\medskip
By \cite[Proposition 2.5]{Miasnikov-Nikolaev-Ushakov:2014a}
computational properties of $\SSP$ do not depend on the choice of a finite generating set $X$
and, hence, the problem can be abbreviated as $\SSP(G)$.
Also, the same paper provides a variety of examples
of groups with  $\NP$-complete (or polynomial time) subset sum problems.
For instance, $\SSP$ is $\NP$-complete for the following groups:
\begin{enumerate}[(a)]
\item 
the abelian group $\mathbb Z^\omega$;
\item 
free metabelian non-abelian groups;
\item 
wreath products of finitely generated infinite abelian groups;
\item 
metabelian Baumslag--Solitar groups $BS(m, n)$ with $0\ne m\ne n\ne 0$;
\item 
the metabelian group
$GB =\gpr{a,s,t}{[a, a^t ] = 1, [s, t] = 1, as = aa^t}$;
\item 
Thompson's group $F$.
\end{enumerate}
One can observe that in a number of the above examples, $\NP$-completeness of $\SSP$ is 
a consequence of exponential subgroup distortion. Further, it is established in~\cite{Nikolaev-Ushakov:2017} that the latter is a sole source of $\NP$-hardness in the case of polycyclic groups. In the present note we show that the $\NP$-hardness of the subset sum problem for weakly regular branch groups is due to existence of abelian subgroups of arbitrarily large rank.

\subsection{Zero-one equation problem}\label{sub:zoe}

Recall that a vector $\overline{v} \in \mathbb Z^n$ is called a {\em zero-one} vector if each entry in $\overline{v}$ is either $0$ or $1$.
Similarly, a square matrix  $A\in \Mat(n,\mathbb Z)$ is called a {\em zero-one} matrix if each entry in $A$ is either $0$ or $1$.
Let $1^n$ denote the vector $(1,\ldots,1) \in \mathbb Z^n$.
The following problem is $\NP$-complete  
(see \cite[Section 8.3]{Dasgupta-Papadimitriou-Vazirani:2006}).

\medskip\noindent
{\bf Zero-one equation problem (ZOE):}
Given  $n$ zero-one vectors $\overline{a_1},\ldots, \overline{a_n}\in\mathbb Z^n$, decide if there exists a zero-one vector $\ovx=(x_1,\ldots,x_n) \in \mathbb Z^n$ satisfying $x_1\overline{a_1}+\cdots+ x_n\overline{a_n}=(1,1,\ldots,1)$, or not.

\subsection{Preliminary result in branch groups} The class of branch groups was originally explicitly defined by Grigorchuk in 1997. Groups in this class possess remarkable algebraic, geometric, and analytic properties, and are studied in relation to just-infiniteness, Burnside problems, random walks, amenability, and many other topics. Geometrically, branch groups are defined in terms of action on rooted trees. We refer the reader to~\cite{Bartholdi-etal:2003} for historic details and a thorough introduction of this class. For purposes of the present paper, we follow terminology exhibited in~\cite{Bartholdi-etal:2003}.

Let a finitely generated 
branch group $G$ act on a regular tree $\mathcal T^{(m)}$, $m\ge 2$. Let $\mathcal L_n$, $n=0,1,2,\ldots$, denote the $n$-th level of $\mathcal T^{(m)}$. Let $\psi$ be the usual embedding of the level $1$ stabilizer into $G^m$, $\psi: \mathop{\mathrm{St}}(\mathcal L_1)\to G^m$.
Recall that a 
branch group $G$ acting on the regular tree $\mathcal T^{(m)}$ is a regular (resp. weakly regular) branch group if $\psi$ is subdirect and there exists a finite index subgroup $K$ of $G$ such that $K^m$ is contained in $\psi(K)$ as a subgroup of finite (resp. perhaps infinite) index. We denote the arising embedding of $K^m$ into $K$ by $\chi$.

Let $\sigma_j$, $j=0,1,\ldots,m-1$, be the embedding $\sigma_j:K\to K^m$, $x\mapsto (1,\ldots,1,x,1,\ldots,1)$, where in the right hand side $x$ is in $(j+1)$-th coordinate. This gives us $m$ embeddings $\varphi_j=\chi\circ \sigma_j:K\to K$, $j=0,\ldots,m-1$.



One can notice that a (weakly) regular branch group contains $\mathbb Z^\infty$ or $\mathbb Z_k^\infty$ as a subgroup. In the next lemma we observe that there is such a subgroup whose first $n$ generators can be produced in polynomial time. We note that a similar construction is employed in~\cite[Section 10]{Bartholdi-etal:2019} (see Lemma~54 and on).

\begin{lem}\label{le:abelian_subgroup}
Let a finitely generated group $G$ be a weakly regular branch group over $K$. There is
\begin{itemize}
\item $k$ which is an integer $k>2$ or infinity,
\item a sequence $a_1,a_2,\ldots\in K$ of group elements of order $k$ such that the sum $\langle a_1\rangle+\langle a_2\rangle+\cdots\le G$ is direct, and
\item a polynomial time algorithm that, given a (unary) positive integer $n$, produces $n$ elements $a_1,\ldots,a_n\in K$.
\end{itemize}
\end{lem}
\begin{proof}
Observe that $K$ has at least one element, say $d$, of infinite order or of order $k>2$, otherwise $K$ is abelian and therefore $G$ is virtually abelian, which is imposible (see, for example,~\cite[Lemma 2]{Grigorchuk-Wilson:2003}).


Let $p$ be the smallest integer such that $2^{p+1}-1\ge n$. Note $p\le \log_2 n$.
Consider the following $1+2+\ldots+2^p\ge n$ tuples of indices:
\begin{align*}
&0,\\
&100, 101,\\
&11000, 11001, 11010, 11011,\\
&\ldots,\\
&\underbrace{1\ldots1}_j0 i_1\ldots i_j,\qquad  i_1,\ldots,i_j=0,1,\\
&\ldots,\\
&\underbrace{1\ldots1}_{p}0 i_1\ldots i_p,\qquad  i_1,\ldots,i_p=0,1.
\end{align*}
For each tuple $i_1\ldots i_\ell$ above, apply the composition
$\varphi_{i_1\ldots i_\ell}=\varphi_{i_1}\circ\cdots\circ\varphi_{i_\ell}$ to the element $d\in K$. We may assume that each $\varphi_j$ is given in terms of (finitely many) generators of $K$, and therefore straightforward computation of each element $a_{i_1\ldots i_\ell}=\varphi_{i_1\ldots i_\ell}(d)$ takes polynomial time, since $\ell\le 2p+1\le 2\log_2n+1$. Since the sum $\varphi_0(K)+\varphi_1(K)\le K$ is direct, it follows that the $2^{p+1}-1$ elements $a_{i_1\ldots i_\ell}$ generate cyclic subgroups whose sum is direct.
\end{proof}

\section{SSP in $\mathbb Z_k^\infty$}\label{se:zoe_to_z_infty}

In this section we consider the infinitely generated group $\mathbb Z_k^\infty$. For algorithmic purposes, we assume that generating elements are encoded by binary strings (see, for example,~\cite[Section 4]{Miasnikov-Nikolaev-Ushakov:2014b}).
\begin{prop}\label{pr:reduction}
Let integer $k\ge 2$. The following holds.
\begin{itemize}
\item 
If $k=2$, then $\SSP(\mathbb Z_k^\infty) \in \P$.
\item 
If $k>2$, then $\SSP(\mathbb Z_k^\infty)$ is $\NP$-complete.
\end{itemize}
\end{prop}

\begin{proof}
If $k=2$, then an instance $(\xi_1,\ldots,\xi_n,\xi)$ of $\SSP(\mathbb Z_k^\infty)$ is positive
if and only if $\xi \in \gp{\xi_1,\ldots,\xi_n}$. The latter can be easily checked using linear algebra.

Let $k>2$. We claim that $\ZOE$ can be reduced to $\SSP(\mathbb Z_k^\infty)$.
Indeed, consider an instance $(\overline{u_1},\ldots,\overline{u_n})$ of $\ZOE$, where
\[
\overline{u_i} = (u_{i1}, \ldots, u_{in}) \mbox{ for each }i=1,\ldots n,
\]
with $u_{ij}\in\{0,1\}$.
Let $b_0 \in \mathbb Z_k^n$ be a sequence of zeros.
For $i=1,\ldots, n$ define a sequence $b_i \in \mathbb Z_k^n$ as a sequence of
zeros with $1$ in $i$th place.
For each $1\le i\le n$ and $v\in\{0,1\}$ define:
\[
b_{iv} = 
\begin{cases}
b_0 & \mbox{if }v=0;\\
b_i & \mbox{if }v=1.\\
\end{cases}
\]
Let $\xi_i$ be a concatenation $b_{i,u_{i1}}\ldots b_{i,u_{in}}$ and $\xi$ a concatenation
$b_{n1}\ldots b_{n1}$.
Also, define $\delta_i \in \mathbb Z_k^n$ (for $1\le i\le n-1$) to be a sequence of zeros except
for $-1$ in $i$th place and $1$ in $(i+1)$th place. Finally, for each $1\le i\le n$ 
and $1\le j\le n-1$ define a  sequence $\delta_{ij}$ to be concatenation of $n-1$ copies of $b_0$
and a single copy of $\delta_j$ in $i$th place:
\[
\delta_{ij} = b_0 \ldots b_0 \delta_j b_0 \ldots b_0.
\]
It is easy to see that if $(\overline{u_1},\ldots,\overline{u_n})$ is a positive instance  of $\ZOE$ then $(\xi_1,\ldots,\xi_n,$ $\delta_{11},\delta_{12},\ldots,$ $\delta_{n,n-1},\xi)$
is a positive instance of $\SSP(\mathbb Z_k^\infty)$. Conversely, suppose the latter is a positive instance of $\SSP(\mathbb Z_k^\infty)$. Inspecting the first $n$ coordinates we observe that in the solution to this instance of $\SSP$, there must be exactly one $\xi_i$ with a $1$ among the first $n$ coordinates; same for the second $n$ coordinates, and so on. It follows that the corresponding tuple $(\overline{u_1},\ldots,\overline{u_n})$ is a positive instance  of $\ZOE$.



Therefore, $\SSP(\mathbb Z_k^\infty)$ is $\NP$-hard when $k>2$. 
Since $\SSP(G)\in\NP$ for every group $G$ with polynomial time word problem we get the result.
\end{proof}

\begin{exa}
Here we give a particular example of the reduction
described above.
Consider an instance of $\ZOE$ with $n=3$:
\[
\begin{array}{ccc}
(1, & 1, & 0), \\
(1, & 0, & 1), \\
(0, & 1, & 0). \\
\end{array}
\]
Then the corresponding instance of 
$\SSP(\mathbb Z_3^\infty)$ is defined by a system of 
sequences with $\ldots$ standing for an infinite 
sequence of zeros:
\[
\begin{array}{lccc|ccc|cccc}
\xi_1= & 1 & 0 & 0 & 1 & 0 & 0 & 0 & 0 & 0&\ldots \\
\xi_2= & 0 & 1 & 0 & 0 & 0 & 0 & 0 & 1 & 0&\ldots \\
\xi_3= & 0 & 0 & 0 & 0 & 0 & 1 & 0 & 0 & 0&\ldots \\
\delta_{11}= & 2 & 1 & 0 & 0 & 0 & 0 & 0 & 0 & 0&\ldots \\
\delta_{12}= & 0 & 2 & 1 & 0 & 0 & 0 & 0 & 0 & 0&\ldots \\
\delta_{21}= & 0 & 0 & 0 & 2 & 1 & 0 & 0 & 0 & 0&\ldots \\
\delta_{22}= & 0 & 0 & 0 & 0 & 2 & 1 & 0 & 0 & 0&\ldots \\
\delta_{31}= & 0 & 0 & 0 & 0 & 0 & 0 & 2 & 1 & 0&\ldots \\
\delta_{32}= & 0 & 0 & 0 & 0 & 0 & 0 & 0 & 2 & 1&\ldots \\
\hline
\xi  = & 0 & 0 & 1 & 0 & 0 & 1 & 0 & 0 & 1&\ldots \\
\end{array}
\]
\end{exa}

\section{Subset sum problem in weakly regular branch groups}





\begin{thm}
Let $G$ be a finitely generated weakly regular branch group. Then $\SSP(G)$ is $\NP$-hard.
\end{thm}
\begin{proof}
By Lemma~\ref{le:abelian_subgroup}, $G$ contains a
subgroup isomorphic to $\mathbb Z^\infty$ or $\mathbb Z_k^\infty$ ($k\in\mathbb Z, k>2$). Recall that $\SSP(\mathbb Z^\infty)$ is $\NP$-complete by \cite{Miasnikov-Nikolaev-Ushakov:2014a}, and $\SSP(\mathbb Z_k^\infty)$, $k\in\mathbb Z, k>2$, is $\NP$-complete by Proposition~\ref{pr:reduction}. By Lemma~\ref{le:abelian_subgroup} it follows that either of those problems is $\P$-time reducible to $\SSP(G)$, therefore $\SSP(G)$ is $\NP$-hard.
\end{proof}

The above theorem applies, for example, to the first Grigorchuk group and all so-called Grigorchuk--Gupta--Sidki groups (see~\cite{Baumslag:93} for a definition).

Since contracting automaton groups have polynomial time decidable word problem~\cite{Nekrashevych:2005}, we obtain the following corollary.

\begin{cor}
Let $G$ be a finitely generated weakly regular contracting branch group. Then $\SSP(G)$ is $\NP$-complete.
\end{cor}
In particular, we note that the first Grigorchuk group has an $\NP$-complete subset sum problem.

As a final remark, we recall that the Lamplighter group also has an $\NP$-complete subset sum problem by~\cite{Misch-Tr:2018, Ganardi-etal:2018}, and the technique used in the proof of that result also involves reduction of $\ZOE$ (more precisely, the easily equivalent Exact Set Cover problem) exploiting ``wide'' abelian subgroups. Since both weakly regular groups and the Lamplighter group are automaton groups, this suggests the following question.

\medskip\noindent
{\sc Question.} Describe which automaton groups have an $\NP$-hard subset sum problem, and which have a polynomial time subset sum problem.

\bibliographystyle{plain}
\bibliography{do_biblio}

\end{document}